\documentclass[a4paper,11pt]{article}
\usepackage{hyperref}
\usepackage{amsmath}
\usepackage{amssymb}
\usepackage{amsfonts}
\usepackage{graphicx}
\usepackage{graphics}
\usepackage{amsthm}
\usepackage{algorithmic}
\usepackage[ruled]{algorithm2e}
\usepackage{multirow,slashbox}
\usepackage{epsfig}
\usepackage{color}
\usepackage{setspace}
\usepackage{mathrsfs}
\usepackage{enumerate}
\usepackage[round,sort&compress]{natbib}
\usepackage[toc,page,title,titletoc,header]{appendix}
\newtheorem{remark}{Remark}[section]

\newtheorem{lemma}{Lemma}[section]
\newtheorem{theorem}{{\sc Theorem}}[section]

\allowdisplaybreaks[0]
\newbox\TempBox \newbox\TempBoxA
\numberwithin{equation}{section}
\begin{document}
\begin{spacing}{1.5}
\title{\bf{Nonparametric estimation of the conditional density function with right-censored and dependent data\thanks{This work was supported by National Natural Science
Foundation of China (No.11301084), and Natural Science
Foundation of Fujian Province, China (No. 2014J01010).}}}
\author{{Xianzhu Xiong\thanks{Corresponding Author. Email: xiongxianzhu2001@sina.com }\,\,,Meijuan Ou}.\\
{College  of Mathematics and Computer Science,}\\
       {Fuzhou University, Fuzhou 350108, China}}
\date{}
\maketitle

\noindent
$\mathbf{Abstract}$\quad In this paper, we study the local constant and the local linear estimators of the conditional density function with right-censored data which exhibit some type of dependence. It is assumed that the observations form a stationary $\alpha-$mixing sequence. The asymptotic  normality of the two estimators is established, which combined with the condition that $\lim\limits_{n\rightarrow\infty}nh_nb_n=\infty$ implies the consistency of the two estimators and can be employed to construct confidence intervals for the conditional density function. The result on the local linear estimator of the conditional density function in Kim et al. (2010) is relaxed from the i.i.d. assumption to the $\alpha-$mixing setting, and the result on the local linear estimator of the conditional density function in Spierdijk (2008) is relaxed from the $\rho$-mixing assumption to the $\alpha-$mixing setting. Finite sample behavior of the estimators is investigated by simulations.
\\\
$\mathbf{Keywords}$\quad conditional density function; right-censored data; $\alpha-$mixing; local constant and local linear estimators; asymptotic normality
\\\
$\mathbf{Mathematics\,\,Subject\,\,Classification}$\quad 62G07; 62N02.

\section{Introduction}
\noindent
It is well-known that the conditional density function plays an important role in the nonparametric statistical inference. It is a good tool not only to uncover the complex relationships between a response variable and some potential covariates, but also to estimate some data characteristicses including the conditional mode and the conditional hazard rate. It has also some applications in financial econometrics (see A\"{\i}t-Sahalia (1999)). The estimation of the conditional density function has been widely studied in the case of multivariate data (see Hyndman et al. (1996), Fan and Yim (2004), Izbicki and Lee (2017)), and it has also received much attention in the case of functional (i.e., infinite dimensional) data (see Ferraty et al. (2005), Rachdi et al. (2014)). However, in these papers, it is assumed that the observations are complete.

In Survival Analysis or Reliability, right-censored data are often encountered. There is substantial literature on nonparametric modelling of right-censored data. For example, Guessoum and Ould Sa\"{\i}d (2008, 2010, 2012) and Liang and Iglesias-P$\acute{\rm e}$rez (2018) considered the estimation of the conditional mean function. This paper focuses on the estimation of the conditional density function, and there are some papers on this topic such as Spierdijk (2008), Kim et al. (2010), Liang and Peng (2010) and Khardani and Semmar (2014). Spierdijk (2008) studied one version of local linear estimators of the conditional density function with stationary $\rho-$mixing observations, while Kim et al. (2010) proposed another version of local linear estimators with different weights in the independent and identical distributed ({i.i.d.}) case. The weights are random quantities determined by the Kaplan-Meier estimation of the survival function of the censoring times, and they indeed include the equal weights used by Spierdijk (2008). Those weights have been also adopted by many papers including Guessoum and Ould Sa\"{\i}d (2008, 2010, 2012), Liang and Peng (2010) and Liang and de U$\rm{\tilde{n}}$a-$\rm{\acute{A}}$lvarez (2011). Especially, Liang and Peng (2010) derived the asymptotic normality and the Berry-Esseen type bound for the kernel estimator of the conditional density function with stationary $\alpha$-mixing observations.

It should be noted that the kernel estimator proposed by Liang and Peng (2010) is a single-smoothing estimator (smoothing the covariates only) of the conditional density function, while the local linear estimator proposed by Kim et al. (2010) is a double-smoothing estimator (smoothing both the response variable and the covariates). In fact, compared with the single-smoothing estimator, the double-smoothing estimator not only appears closer to a conditional density function, but also has more flexibility to reduce the mean-squared error when the optimal bandwidths are selected. However, the asymptotic distribution of the local linear estimator of the conditional density function in Kim et al. (2010) was investigated in the i.i.d. setting (although Kim et al. (2010) established the asymptotic distribution of the local linear estimator of the conditional hazard function, the asymptotic distribution of the local linear estimator of the conditional density function is easily obtained by the same method).

The dependent data scenario is an important one in many applications with survival data. For example, in the domain of clinical trials, it often happens that the patients from the same hospital have correlated survival times due to unmeasured variables like the quality of the hospital equipment, and an example of such data can be found in Lipsitz and Ibrahim (2000). Correlated and censored data frequently appear in the domain of spatial and environmental statistics, and an example of such data can be found in Eastoe et al. (2006). Then, it is significant to study the asymptotic properties of the local linear estimator of the conditional density function with right-censored and dependent data.

In this paper, we establish the asymptotic normality of the local constant and the local linear estimators of the conditional density function for a right-censored model in the $\alpha$-mixing setting. The property of asymptotic normality combined with the condition that $\lim\limits_{n\rightarrow\infty}nh_nb_n=\infty$ implies the consistency of the two estimators and can be employed to construct confidence intervals for the conditional density function. There are three minor contributions in this paper. Firstly, Theorem 3.2 establishes the asymptotic normality of the local linear estimator, which extends the asymptotic distribution of the local linear estimator of the conditional density function in Kim et al.(2010) from the i.i.d. assumption to the $\alpha$-mixing setting. Since that $\alpha$-mixing is weaker than $\rho-$mixing and the fact that the local linear estimator in Spierdijk (2008) is a special case of the local linear estimators in Kim et al. (2010), Theorem 3.2 also extends the corresponding result in Spierdijk (2008). Secondly, Theorem 3.1 establishes the asymptotic normality of the local constant estimator (i.e., the Nadaraya-Watson estimator) of the conditional density function. The local constant estimator is actually a smoothing estimator of the kernel estimator proposed by Liang and Peng (2010). It follows from Theorem 3.1 and Theorem 3.2 that the local liner estimator has better asymptotic bias than the local constant estimator, which is the same as that in the case of complete data. And the simulations also further confirm it. Thirdly, we adopt the bandwidth condition used by Liang and Peng (2010), which is weaker than that in Kim et al. (2010).

Liang and Baek (2016) established the asymptotic distribution of the local constant and local linear estimators of the conditional density function for a left-truncation model in the $\alpha$-mixing setting. As pointed out by Stute (1993), right-censored data is completely different from left-truncated data, then the results for right-censored data cannot be deduced from those obtained in the left-truncated case. The rest of this paper is organized as follows. Section 2 describes the mixing condition and the nonparametric estimators. Assumptions and the main results of the estimators are stated in Section 3. A simulation study is presented in Section 4. Section 5 gives the proofs of the main results, while the proofs of some lemmas are postponed in Section 6.

\section{Nonparametric estimator with dependent right-censored data}
\subsection{Dependence structure}
\noindent
A sequence $\{Z_i,i\geq1\}$ is called $\alpha-$mixing or strongly mixing, if the $\alpha-$mixing coefficient
\begin{displaymath}
\alpha(n)=:\sup_{k\geq 1}\sup_{A\in\mathcal{F}_{1}^{k},B\in\mathcal{F}_{k+n}^{+\infty}}|P(AB)-P(A)P(B)|
\end{displaymath}
converges to $0$ as $n\rightarrow\infty$, where $\mathcal{F}_{j}^{k}=\sigma(\{Z_i|j\leq i \leq k\})$ denotes the $\sigma-$algebra generated by $\{Z_i|j\leq i\leq k\}$. It is well-known that $\alpha-$mixing is the weakest among various mixing conditions available in the literature, and many processes do fulfill the $\alpha-$mixing property. The reader can see Doukhan (1994) for a more complete discussion of the $\alpha-$mixing property.

\subsection{Estimators}
Let $T$ be a right-censored observation such that $T=\min\{Y,C\}$, where $Y$ is the survival time with a continuous distribution function (d.f.) $F(\cdot)$ and $C$ is the censoring time with a continous d.f. $G(\cdot)$ and the corresponding survival function $H(\cdot)$. Then $T$ has the d.f. $Q(u)=1-(1-F(u))(1-G(u)),\,u\in\mathbb{R}$. Let $\delta={\mathbf{I}}_{\{Y \leq C\}}$, where ${\mathbf{I}}_A$ denotes the indicator function of the set $A$. It is supposed that the associated covariate $X$ is a real-valued covariate  with a density function $f_X(\cdot)$. Moreover, for simplicity, we assume that $(X,Y)$ and $C$ are independent. The assumption is also adopted by many papers such as Guessoum and Ould Sa\"{\i}d (2008, 2010, 2012), Liang and Peng (2010) and Liang and de U$\rm{\tilde{n}}$a-$\rm{\acute{A}}$lvarez (2011). As pointed out by Kim et al. (2010), this assumption is a reasonable assumption since censoring occurs due to termination of study in most applications. Let $f_{(X,Y)}(\cdot,\cdot)$ be the joint density function of  $(X,Y)$, then for $x\in supp(f_X)=\{u\in\mathbb{R}|f_X(u)>0\}$, the conditional density function of $Y$, given $X=x$, is
\begin{align*}
f(y|x)=\frac{f_{(X,Y)}(x,y)}{f_X(x)}, \quad y\in\mathbb{R}.
\end{align*}

It is supposed that $\{(X_i,Y_i,C_i)|i\geq 1\}$ is a stationary $\alpha-$mixing sequence with coefficient $\alpha_1(n)$ from $(X,Y,C)$. Then, it follows from Lemma 2 of Cai (2001) that $\{(X_i,T_i,\delta_i)| i\geq 1\}$ is a stationary $\alpha-$mixing sequence with coefficient $4\alpha_1(n)$. From now on, it is supposed that $\{(X_i,T_i,\delta_i)| i\geq 1\}$ is a stationary $\alpha-$mixing sequence with coefficient $\alpha(n)$.

Suppose that the second derivative of  $f(y|s)$ with respect to $s$ at the point $x$ exists. Recall that, using the observations $\{(X_i,T_i,\delta_i)|1\leq i\leq n\}$,  Kim et al. (2010) proposed the local linear estimator $(\hat{f}_{LL}(y|x),{\hat{f}}^{(1,0)}_{LL}(y|x))^{\tau}$  of $(f(y|x),{f}^{(1,0)}(y|x))^{\tau}$, which is defined as $(\hat{a},\hat{b})$, where
\begin{alignat}{1}
(\hat{a},\hat{b})=\arg\min\limits_{(a,b)\in\mathbb{R}^2}\sum_{i=1}^{n}\Big(\frac{\delta_i}{\hat{H}(T_i)}\Lambda_{b_n}(T_i-y)-a-b(X_i-x)\Big)^2K_{h_n}(X_i-x),
\end{alignat}
where $f^{(i,j)}(y|x)=\partial^{(i+j)}f(y|x)/\partial x^i\partial y^j$, $K(\cdot)$ and $\Lambda(\cdot)$ are both  kernel functions, $K_{h_n}(\cdot)=K(\cdot/h_n)/h_n$, $\Lambda_{b_n}(\cdot)=\Lambda(\cdot/b_n)/b_n,$  $0<h_n\rightarrow0 (as\, n\rightarrow\infty)$, $0<b_n\rightarrow0 ( as\, n\rightarrow\infty)$, and $\hat{H}(\cdot)$ is the Kaplan-Meier estimator of $H(\cdot)$, which is defined as
\begin{equation*}
\hat{H}(t)=
\begin{cases}
\prod\limits_{i:T_{(i)} \leq t}\big(\frac{n-i}{n-i+1}\big)^{1-\delta_{(i)}},& \text{if}\quad t< T_{(n)},\\
0,&\text{if}\quad t\geq T_{(n)},
\end{cases}
\end{equation*}
where $T_{(1)}\leq T_{(2)}\leq \cdots\leq T_{(n)}$ are the order statistics of $T_1, T_2, \cdots, T_n$ and $\delta_{(i)}$ is the concomitant of $T_{(i)}$. By simple algebra, the local linear estimator $(\hat{f}_{LL}(y|x),{\hat{f}}^{(1,0)}_{LL}(y|x))^{\tau}$ can be explicitly written as
\begin{align}(\hat{f}_{LL}(y|x),{\hat{f}}^{(1,0)}_{LL}(y|x))^{\tau}=(\mathbf{X}^{\tau}\mathbf{W}\mathbf{X})^{-1}\mathbf{X}^{\tau}\mathbf{W}\mathbf{T},\end{align} where
${\mathbf{X}=\left(\begin{array}{c c}
1& (X_1-x)\\
\vdots & \vdots \\
1& (X_n-x)
\end{array}\right)}$, ${\mathbf{T}=\left(\begin{array}{c }
\frac{\delta_1}{\widehat{H}(T_1)}\Lambda_{b_n}(T_1-y) \\ \vdots \\ \frac{\delta_n}{\widehat{H}(T_n)}\Lambda_{b_n}(T_n-y)
\end{array}\right)}$, $\mathbf{W}=\text{diag}(K_{h_n}(X_i-x))$.
Let $\mathbf{t}_n={\Bigg(\begin{array}{c}t_{n0}(x) \\ t_{n1}(x)\end{array}\Bigg)}$ and $\mathbf{S}_n={\Bigg(\begin{array}{cc}s_{n0}(x) & s_{n1}(x)\\ s_{n1}(x)& s_{n2}(x)\end{array}\Bigg)}$, where $s_{nj}(x)=\frac {1}{n}\sum\limits_{i=1}^{n}\Big(\frac{X_i-x}{h_n}\Big)^jK_{h_n}(X_i-x),\,j=0,1,2;$ and\\
$t_{nj}(x)=\frac
  {1}{n}\sum\limits_{i=1}^{n}\Big(\frac{X_i-x}{h_n}\Big)^jK_{h_n}(X_i-x)\frac{\delta_i}{\widehat{H}(T_i)}\Lambda_{b_n}(T_i-y),\,j=0,1.$
 Then, $(\hat{f}_{LL}(y|x),{\hat{f}}^{(1,0)}_{LL}(y|x))^{\tau}$ can also be written as
 \begin{align}(\hat{f}_{LL}(y|x), \hat{f}^{(1,0)}_{LL}(y|x))^{\tau}=\text{diag}(1,h^{-1}_n)\mathbf{S}^{-1}_n\mathbf{t}_n.\end{align}

Similarly, the local constant estimator $\hat{f}_{NW}(y|x)$  of $f(y|x)$  can be defined as $\hat{c}$, where
\begin{align*}
\hat{c}=\arg\min\limits_{c\in\mathbb{R}}\sum_{i=1}^{n}\Big(\frac{\delta_i}{\hat{H}(T_i)}\Lambda_{b_n}(T_i-y)-c\Big)^2K_{h_n}(X_i-x).
\end{align*}
By simple algebra, $\hat{f}_{NW}(y|x)$ can be explicitly written as
\begin{alignat}{1}
\hat {f}_{NW}(y|x)=\frac{\frac{1}{n}\sum\limits_{i=1}^{n}K_{h_n}(X_i-x)\frac{\delta_i}{\widehat{H}(T_i)}\Lambda_{b_n}(T_i-y)}
{\frac{1}{n}\sum\limits_{i=1}^{n}K_{h_n}(X_i-x)},
\end{alignat}
which is actually a smoothing estimator of the kernel estimator proposed by Liang and Peng (2010).

\section{Assumptions and the main result}
\noindent In what follows, let $N(x)$ denote a neighborhood of $x$, and let $Const$ denote some finite and positive constant whose value may change from place to place. Now give the following assumptions needed to obtain our results and they are gathered here for convenient reference. The notations used in this section are similar to those used by Liang and Baek (2016).\\
\noindent
(A1) Both $K(\cdot)$ and $\Lambda(\cdot)$  are symmetric probability density functions with compact support $[-1,1]$.\\
(A2) The second partial derivatives of $f(\cdot|\cdot)$ are continuous in $N(x)\times N(y)$, and $f(y|x)>0$.\\
(A3) $f_X({\cdot})$ has bounded second derivatives in $N(x)$ and $f_X(x)>0.$\\
(A4)(\textrm{i}) For every positive integer $k$, there is the joint density function $f_k(\cdot,\cdot)$ of $(X_1,X_{1+k})$ on $\mathbb{R}\times\mathbb{R}$,\\ \hspace*{1.2cm}which satisfies $f_k(u,v)\leq Const$  for $(u,v)\in N(x)\times N(x)$ .\\
\hspace*{0.8cm}(\textrm{ii}) For every positive integer $k$, there is the joint density function $f_k(\cdot,\cdot,\cdot)$ of $(X_1,X_{1+k},Y_1)$ on $\mathbb{R}\times\mathbb{R}\times\mathbb{R}$, \\
\hspace*{1.2cm}which satisfies $f_k(u,v,w)\leq  Const$ for $(u,v,w)\in N(x)\times N(x)\times N(y)$.\\
\hspace*{0.8cm}(\textrm{iii}) For every positive integer $k$, there is the joint density function $f_k(\cdot,\cdot,\cdot)$ of $(X_1,X_{1+k},Y_{1+k})$ on\\ \hspace*{1.2cm} $\mathbb{R}\times\mathbb{R}\times\mathbb{R}$,
which satisfies $f_k(u,v,w)\leq  Const$ for $(u,v,w)\in N(x)\times N(x)\times N(y)$.\\
\hspace*{0.8cm}(\textrm{iv}) For every positive integer $k$, there is the joint density function $f_k(\cdot,\cdot,\cdot,\cdot)$ of $(X_1,X_{1+k},Y_1,Y_{1+k})$ on \\ \hspace*{1.2cm} $\mathbb{R}\times\mathbb{R}\times\mathbb{R}\times\mathbb{R}$,
which satisfies  $f_k(u_1,u_2,v_1,v_2)\leq  Const$ for $(u_1,u_2,v_1,v_2)\in N(x)\times N(x)\times N(y)\times N(y)$.\\
(A5) Assume that $\lim\limits_{n\rightarrow\infty}nh_nb_n=\infty,$ and the sequence $\alpha(n)$ satisfies that there exist positive integers $q_n$  such that $q_n=o((nh_nb_n)^{1/2})$ and $\lim\limits_{n\rightarrow \infty}(n(h_nb_n)^{-1})^{1/2}\alpha(q_n)=0$.
\begin{remark}
Assumptions (A1)-(A3) are commonly used  in the literature including Kim et al. (2010). Assumption (A4) is a technical assumption, which is used to make the calculations of covariances simple in the proofs and is needless in the independent setting. Assumption (A5) is often used to prove asymptotic normality of an $\alpha-$mixing process by Doob's technique. In the i.i.d. setting, assumption (A5) is just the assumption that $\lim\limits_{n\rightarrow\infty}nh_nb_n=\infty,$ which is weaker than assumption (C5) in Kim et al. (2010). Furthermore, suppose that $\alpha(n)=O(n^{-\lambda})$ for some $\lambda>0$, then assumption (A5) will be satisfied, for instance, take $h_n=b_n=O(n^{-\eta})$ for some $0<\eta<1/2$, $q_n=(nh_n^2/\log n)^{1/2}$ and $\lambda>(1+2\eta)/(1-2\eta).$ In what follows, it is assumed that $\alpha(n)=O(n^{-\lambda})$ for some $\lambda>0$.
\end{remark}

 Put $\triangle_{ij}=\int_\mathbb{R} u^iK^j(u)du$,$\nabla_{ij}=\int_\mathbb{R} u^i\Lambda^j(u)du$, where $i,j=0,1,2,3$;
$\mathbf{S}={\Bigg( \begin{array}{cc}
1 &0\\
0&\triangle_{21}
\end{array}
\Bigg )}$, $\mathbf{V}={\Bigg( \begin{array}{cc}
\triangle_{02}&\triangle_{12}\\
\triangle_{12}&\triangle_{22}
\end{array}
\Bigg )}$, and $\mathbf{U}={\Bigg( \begin{array}{c}
\triangle_{21}\\
\triangle_{31}
\end{array}
\Bigg )}$. Let  $\tau_{_Q}=\sup\{y\,|\,Q(y)<1\}$.

\begin{theorem} Let $\alpha(n)=O(n^{-\lambda})$ for some $\lambda>3$. Suppose that assumptions (A1)-(A5) hold, and $y<\tau_{_Q}$. Then
\begin{eqnarray*}
\sqrt{nh_nb_n}\Big\{\hat{f}_{NW}(y|x)-f(y|x)-\frac{h_n^2}{2}\triangle_{21}\Big[f^{(2,0)}(y|x)+2\frac{f_X^{'}(x)}{f_X(x)}f^{(1,0)}(y|x)\Big]
\\-\frac{b_n^2}{2}\nabla_{21}f^{(0,2)}(y|x)+o_{p}(h_n^2+b_n^2)\Big\}\stackrel{D}{\longrightarrow}
N\big(0,\sigma^2(y|x)\triangle_{02}\big),
\end{eqnarray*}
where $\sigma^2(y|x)=\frac{f(y|x)}{H(y) f_X(x)}\nabla_{02}.$
\end{theorem}

\begin{theorem} Let $\alpha(n)=O(n^{-\lambda})$ for some $\lambda>3$. Suppose that assumptions (A1)-(A5) hold, and $y<\tau_{_Q}$. Then
\begin{eqnarray*}
\sqrt{nh_nb_n}\Bigg\{{\rm{diag}}(1,h_n){
\Bigg( \begin{array}{c}
\hat{f}_{LL}(y|x)-f(y|x) \\
\hat{f}_{LL}^{(1,0)}(y|x)-f^{(1,0)}(y|x)
\end{array}
\Bigg )}-\frac{h_n^2}{2}f^{(2,0)}(y|x)\mathbf{S}^{-1}\mathbf{U}
\\-\frac{b_n^2}{2}\nabla_{21}f^{(0,2)}(y|x){
\Bigg( \begin{array}{c}
1 \\
0
\end{array}
\Bigg )}+o_p(h_n^2+b_n^2)\Bigg\}\stackrel{D}{\longrightarrow}
N\big(0,\sigma^2(y|x)\mathbf{S}^{-1}\mathbf{V}\mathbf{S}^{-1}\big).
\end{eqnarray*}
\end{theorem}

\begin{remark}
Theorem 3.1 and Theorem 3.2 are similar to Theorem 3.1 and Theorem 3.2 in Liang and Baek (2016), respectively. But as mentioned before, Theorem 3.1 and Theorem 3.2 can not be deduced from Theorem 3.1 and Theorem 3.2 in Liang and Baek (2016), respectively.
\end{remark}
\begin{remark}It follows from Theorem 3.1 and Theorem 3.2 that, if $nh^{5}_nb_n\rightarrow Const$ and $nh_nb^{5}_n\rightarrow Const$ for some $Const\neq 0$,
\begin{align*}
&\sqrt{nh_nb_n}\Big\{\!\hat{f}_{NW}(y|x)\!-\!f(y|x)\!-\!\frac{h_n^2\triangle_{21}}{2}\Big[f^{(2,0)}(y|x)\!+\!\frac{2f_X^{'}(x)}{f_X(x)}f^{(1,0)}(y|x)\Big]
\!-\!\frac{b_n^2\nabla_{21}}{2}f^{(0,2)}(y|x)\!\Big\}\!\stackrel{D}{\longrightarrow}\!
N\big(0,\sigma^2(y|x)\triangle_{02}\big),\nonumber\\
&and\nonumber\\
&\sqrt{nh_nb_n}\Big\{\hat{f}_{LL}(y|x)-f(y|x)-\frac{h_n^2\triangle_{21}}{2}f^{(2,0)}(y|x)
-\frac{b_n^2\nabla_{21}}{2}f^{(0,2)}(y|x)\nonumber\Big\}\stackrel{D}{\longrightarrow}
N\big(0,\sigma^2(y|x)\triangle_{02}\big).
\end{align*}
It is obvious that $\hat{f}_{NW}(y|x)$ and $\hat{f}_{LL}(y|x)$ have the same asymptotic variance, while the asymptotic bias of $\hat{f}_{NW}(y|x)$ is larger than that of $\hat{f}_{LL}(y|x)$ when $|\frac{f_X^{'}(x)}{f_X(x)}|$ or $|f^{(1,0)}(y|x)|$ is large. In fact, the term $|\frac{f_X^{'}(x)}{f_X(x)}|$ becomes large in the case that the design density is highly clustered.
\end{remark}

\begin{remark}If $nh^{5}_nb_n\rightarrow 0$ and $nh_nb^{5}_n\rightarrow 0$, it follows from Theorem 3.1 and Theorem 3.2 that,
\begin{align*}
\sqrt{nh_nb_n}\Big\{\hat{f}_{NW}(y|x)\!-\!f(y|x)\Big\}\!\stackrel{D}{\longrightarrow}\!
N\big(0,\sigma^2(y|x)\triangle_{02}\big)\,\,and\,\,\sqrt{nh_nb_n}\Big\{\hat{f}_{LL}(y|x)\!-\!f(y|x)\Big\}\!\stackrel{D}{\longrightarrow}\!
N\big(0,\sigma^2(y|x)\triangle_{02}\big).
\end{align*}
Now we apply the normal-approximation-based method to construct confidence intervals for the conditional density $f(y|x)$. The unknown parameter $f_X(x)$ can be estimated by the usual kernel estimator given by $\hat{f}_X(x)=\frac{1}{nh_n}\sum\limits_{i=1}^{n}K\Big(\frac{X_i-x}{h_n}\Big)$. And the two parameters $H(y)$ and $f(y|x)$ can be estimated by $\hat{H}(y)$
and $\hat{f}_{NW}(y|x)$, respectively. Hence, we can get a plug-in estimator for the asymptotic variance $\sigma^2(y|x)\triangle_{02}$. Thus we can construct asymptotic confidence intervals of $f(y|x)$.
\end{remark}

\section{Simulation Study}
In this section, a simulation study is carried out to investigate the finite sample performance of the local linear estimator $\hat{f}_{LL}(y|x)$ of the conditional density function under right-censored and dependent data. To be specific, we compare the performance among the local linear estimator $\hat{f}_{LL}(y|x)$, the Nadaraya-Watson estimator $\hat{f}_{NW}(y|x)$ and the kernel estimator denoted by $\hat{f}_{K}(y|x)$ (proposed by Liang and Peng (2010)) by their global mean squared errors (GMSE). In order to get an $\alpha-$mixing observed sequence $\{(X_i,T_i,\delta_i)|\,1\leq i\leq n\}$ after censoring, we generate the observed data as follows.

First generate a sequence of covariate $\{X_i|1\leq i\leq n\}$ as follows: $X_1\sim N(1,0.5^2),\,X_i=\rho X_{i-1}+e_i, i=2,\cdots,n,$ where $0<\rho<1,\,\{e_i|2\leq i\leq n\}\stackrel{{\rm{i.i.d.}}}{\sim}N(1-\rho,(0.5\sqrt{1-\rho^2})^2)$ and it is independent of the sequence $\{X_i|1\leq i\leq n\}$. Generate a sequence of response $\{Y_i|1\leq i\leq n\}$ as follows: $\log Y_i=X_i+\epsilon_i, i=1,2,\cdots,n,$ where
$\{\epsilon_i|1\leq i\leq n\}\stackrel{{\rm{i.i.d.}}}{\sim}N(0,1)$ and it is independent of the sequence $\{X_i|1\leq i\leq n\}$. We also generate an i.i.d. sequence of censoring time $\{C_i|1\leq i\leq n\}$, which follows a lognormal distribution with parameters $\mu$ and $0.6^2$ and is independent of the sequences $\{X_i|1\leq i\leq n\}$ and $\{Y_i|1\leq i\leq n\}$. As in Kim et al. (2010) and Spierdijk (2008), $\mu$ is chosen in such a way that the expected censoring percentages (CP) are $10\%, 30\%$ and $50\%$.

Then the sequence $\{X_i|1\leq i\leq n\}$ is a stationary $\alpha-$mixing sequence.
It follows from the $\alpha-$mixing property and the mutual independence among the three sequences $\{X_i|1\leq i\leq n\}$, $\{\epsilon_i|1\leq i\leq n\}$ and $\{C_i|1\leq i\leq n\}$, that $\{(X_i,Y_i,C_i)|1\leq i\leq n\}$ is an $\alpha-$mixing sequence, and the corresponding observed sequence $\{(X_i,T_i,\delta_i)|1\leq i\leq n\}$  is an $\alpha-$mixing sequence, which is actually a dependent version of data generated in Kim et al. (2010). Hence the conditional density function $f(y|x)=\mathbf{I}_{\{y>0\}}\frac{1}{\sqrt{2\pi}y}\exp\big\{-\frac{(\log y-x)^2}{2}\big\}$.

Take $K(u)=\Lambda(u)=\frac{3}{4}(1-u^2)\mathbf{I}_{\{|u|\leq1\}}$ and $\rho=0.3$. For the sample sizes $n=100$, $200$ and $300$, we compute the GMSE for every estimator $\hat{f}(\cdot|x)$ of ${f}(\cdot|x)$ at the point $x=1$ along $M=100$ Monte Carlo trials and a grid of bandwidths $h:=h_n$ and $b:=b_n$. For every estimator $\hat{f}_{LL}(\cdot|x),$ $\hat{f}_{NW}(\cdot|x),$ and $\hat{f}_{K}(\cdot|x)$, the GMSE of  $\hat{f}(\cdot|x)$ is defined as
\begin{align*}GMSE(h,b)=\frac{1}{Mn}\sum\limits_{l=1}^M\sum\limits_{k=1}^n\Big[\hat{f}(T_k,l|x)-{f}(T_k,l|x)\Big]^2.\end{align*}

Table 1 reports the minimal values of $GMSE(h,b)$ along the grid. Table 1 shows that (i) The minimum GMSEs of three estimators decrease as the sample zize increases, and
the minimum GMSEs of three estimators decrease as the CP decreases. This was expected, since we are moving to situations with more sampling information. (ii)  In each case, $\hat{f}_{LL}(\cdot|x)$ has a smaller GMSE than $\hat{f}_{NW}(\cdot|x)$, and $\hat{f}_{NW}(\cdot|x)$ has a smaller GMSE than $\hat{f}_{K}(\cdot|x)$.
\begin{table}
\renewcommand{\arraystretch}{1.4}
\setlength{\abovecaptionskip}{18pt}
\setlength{\belowcaptionskip}{18pt}
\caption{Minimum GMSEs of $\hat{f}_{K}(\cdot|x),$ $\hat{f}_{NW}(\cdot|x),$ and $\hat{f}_{LL}(\cdot|x)$ at $x=1$ for several censoring percentages  }
\label{tab1}
\centering
\begin{tabular}{lccccccc}
\hline
$CP$ \,\,\,\, &\,\,\, $n$ \,\,\,\,\,\, & \,\,\,\,$\hat{f}_{K}(\cdot|x)$ \,\,\,\,\,\, &\,\,$\hat{f}_{NW}(\cdot|x)$\, \,\,\,\,\,\, &\,\,\, $\hat{f}_{LL}(\cdot|x)$  \\
\hline
10\%   &100  & 0.01379&0.01168& 0.01022  \\
      &150  & 0.01303&0.01085& 0.00979 \\
      &200  & 0.01272&0.00978& 0.00906 \\
\hline
30\%    &100  & 0.01583&0.01252& 0.01138 \\
      &150  & 0.01476&0.01155& 0.01053\\
      &200  & 0.01342&0.01070& 0.00955\\
\hline
50\%    &100  & 0.04189&0.02541& 0.02456 \\
      &150  & 0.04159&0.02433& 0.02367\\
      &200  & 0.03926&0.02310& 0.02287\\
\hline
\centering
\end{tabular}
\end{table}

\section{Proofs of the main results}
Let $f_n(y|x)=E[\Lambda_{b_n}(Y-y)|X=x]$, $t^{*}_{nj}(x)=\frac{1}{n}\sum\limits_{i=1}^{n}\Big(\frac{X_i-x}{h_n}\Big)^jK_{h_n}(X_i-x)\big\{\frac{\delta_i}{\hat{H}(T_i)}
\Lambda_{b_{n}}(T_i-y)-f_n(y|X_i)\big\},
$ \\$\mathbf{t}^{*}_{n}=\big(t^{*}_{n0}(x),t^{*}_{n1}(x)\big)^{\tau},$
$t^{*}_{1nj}(x)=\frac{1}{n}\sum\limits_{i=1}^{n}\Big(\frac{X_i-x}{h_n}\Big)^jK_{h_n}(X_i-x)\big\{\frac{\delta_i}{H(T_i)}
\Lambda_{b_{n}}(T_i-y)-f_n(y|X_i)\big\},$
$\mathbf{t}^{*}_{1n}=\big(t^{*}_{1n0}(x),t^{*}_{1n1}(x)\big)^{\tau}.$

\begin{lemma}
Suppose that $\alpha(n)=O(n^{-\lambda})$ for some $\lambda>3.$ Then, for any $\gamma\in(0,\tau_{_Q})$, we have  \begin{align*}\sup\limits_{u\in[0,\gamma]}|\hat{H}(u)-H(u)|= O_p(n^{-1/2}).\end{align*}
\end{lemma}
\begin{remark}
Lemma 5.1 follows from Lemma A.5(i) in Liang and Peng (2010).
\end{remark}

\begin{lemma}
Let  $\alpha(n)\!=\!O(n^{-\lambda})$ for some $\lambda>3.$ Suppose that assumptions (A1), (A3), (A4)(i), and (A5) hold. If $nh_n\!\rightarrow\!\infty$, then $s_{nj}(x)\stackrel{P}{\longrightarrow}f_X(x)\triangle_{j1}$ for $0\leq j\leq3$; moreover, $\mathbf{S}_n\stackrel{P}{\longrightarrow}f_X(x)\mathbf{S}.$
\end{lemma}
\begin{remark}
Lemma 5.2 is a direct corollary of Step 5 in Liang and Baek (2016) in the case that $G(y)\equiv1$ and $\theta=1$.
\end{remark}

\begin{lemma}
 Let  $\alpha(n)=O(n^{-\lambda})$ for some $\lambda>3.$ Suppose that assumptions (A1)-(A5) are satisfied. Then for $y<\tau_{_Q}$,
\begin{eqnarray*}
\sqrt{nh_nb_n}\Big\{\tilde{f}_{NW}(y|x)-f(y|x)-\frac{h_n^2}{2}\triangle_{21}\Big[f^{(2,0)}(y|x)+2\frac{f_X^{'}(x)}{f_X(x)}f^{(1,0)}(y|x)\Big]
\\-\frac{b_n^2}{2}\nabla_{21}f^{(0,2)}(y|x)+o_p(h_n^2+b_n^2)\Big\}\stackrel{D}{\longrightarrow}
N\big(0,\sigma^2(y|x)\triangle_{02}\big),
\end{eqnarray*}
where $\tilde {f}_{NW}(y|x)=\frac{\frac{1}{n}\sum\limits_{i=1}^{n}K_{h_n}(X_i-x)\frac{\delta_i}{{H}(T_i)}\Lambda_{b_n}(T_i-y)}
{\frac{1}{n}\sum\limits_{i=1}^{n}K_{h_n}(X_i-x)}.$

\end{lemma}

\begin{lemma}
Let $\alpha(n)=O(n^{-\lambda})$ for some $\lambda>3$. Suppose that assumptions (A1)-(A5) hold, and $y<\tau_{_Q}$, then
\begin{eqnarray*}
\sqrt{nh_nb_n}{\mathbf{t}}^{*}_{1n}\stackrel{D}{\longrightarrow}
N\big(0,\sigma^2(y|x)f^2_X(x)\mathbf{V}\big).
\end{eqnarray*}
\end{lemma}

The proofs of Lemmas 5.3 and 5.4 are relegated to Section 6.
\begin{proof}[Proof of Theorem 3.1]
Note that
\begin{align*}
\hat{f}_{NW}(y|x)-f(y|x)=\hat{f}_{NW}(y|x)-\tilde{f}_{NW}(y|x)+\tilde{f}_{NW}(y|x)-f(y|x),
\end{align*}
then Theorem 3.1 will follow from Lemma 5.3 if it can be proved that
\begin{align}
\sqrt{nh_nb_n}\Big|\hat{f}_{NW}(y|x)-\tilde{f}_{NW}(y|x)\Big|=o_p(1).
\end{align}

It follows from the fact that $y<\tau_{_Q}$ and $b_n\rightarrow0$ that there is a $\gamma<\tau_{_Q}$ such that $y+b_n\leq\gamma$ for a large $n$. Then, by assumption (A1), Lemma 5.1 and Lemma 5.2,
\begin{align}
\sqrt{nh_nb_n}\Big|\hat{f}_{NW}(y|x)-\tilde{f}_{NW}(y|x)\Big|=&\sqrt{nh_nb_n}\frac{\frac{1}{n}\left|\sum\limits_{i=1}^{n}K_{h_n}(X_i-x)\frac{\delta_i}{H(T_i)}\Lambda_{b_n}(T_i-y)
\frac{[\hat{H}(T_i)-H(T_i)]}{\hat{H}(T_i)}\right|}
{\frac{1}{n}\sum\limits_{i=1}^{n}K_{h_n}(X_i-x)}\nonumber\\
\leq&\frac{\sqrt{nh_nb_n}\sup\limits_{u\in[0,\gamma]} |\hat{H}(u)-H(u)|}{H(\gamma)-\sup\limits_{u\in[0,\gamma]}|\hat{H}(u)-H(u)|}
\times\frac{\frac{1}{n}\sum\limits_{i=1}^{n}K_{h_n}(X_i-x)\frac{\delta_i}{H(T_i)}\Lambda_{b_n}(T_i-y)}{s_{n0}(x)}\nonumber\\
=& \sqrt{nh_nb_n}O_p(n^{-\frac{1}{2}}) O_p(1)\times\frac{1}{n}\sum\limits_{i=1}^{n}K_{h_n}(X_i-x)\frac{\delta_i}{H(T_i)}\Lambda_{b_n}(T_i-y)\nonumber\\
=&O_p((h_nb_n)^\frac{1}{2})\times\frac{1}{n}\sum\limits_{i=1}^{n}K_{h_n}(X_i-x)\frac{\delta_i}{H(T_i)}\Lambda_{b_n}(T_i-y).
\end{align}

 It follows from the independence of $(X,Y)$ and $C$ that
\begin{align*}E\left[\frac{\delta}{H(T)}\Lambda_{b_n}(T-y)\Big|X=x\right]&=E\left[\frac{{\mathbf{I}}_{\{Y \leq C\}}}{H(Y)}\Lambda_{b_n}(Y-y)\Big|X=x\right]\nonumber\\
&=E\left[\frac{\Lambda_{b_n}(Y-y)}{H(Y)}E[{\mathbf{I}}_{\{Y \leq C\}}|Y]\Big|X=x\right]\nonumber\\
&=E[\Lambda_{b_n}(Y-y)|X=x].\end{align*}
Then \begin{align}f_n(y|x)=E[\Lambda_{b_n}(Y-y)|X=x]=E\left[\frac{\delta}{H(T)}\Lambda_{b_n}(T-y)\Big|X=x\right].\end{align}
According to assumptions (A1) and (A2), and by using the second Taylor expansion at the point $(x,y)$ of the function $f(\cdot|\cdot)$ in the computation of $f_n(y|u)\,(u\in N(x))$,
\begin{alignat}{2}
f_n(y|u)\!=\!f(y|x)\!+\!f^{(1,0)}(y|x)(u-x)\!+\!\frac{1}{2}f^{(2,0)}(y|x)(u-x)^{2}\!+\!o((u-x)^{2})
\!+\!\frac{1}{2}b_n^{2}f^{(0,2)}(y|x)\nabla_{21}+o(b_n^{2}).
\end{alignat}
By stationarity, (5.3), (5.4), assumption (A1) and assumption (A3),
\begin{align}
E\Big(\frac{1}{n}\sum\limits_{i=1}^{n}K_{h_n}(X_i-x)\frac{\delta_i}{H(T_i)}\Lambda_{b_n}(T_i-y)\Big)
=&E\left[K_{h_n}(X-x)E\Big[\frac{\delta}{H(T)}\Lambda_{h_n}(T-y)|X\Big]\right]\nonumber\\
=&\int_\mathbb{R}K(v) f_n(y|x+h_nv) f_X(x+h_nv)dv=O(1),
\end{align}
which implies that $\frac{1}{n}\!\sum\limits_{i=1}^{n}K_{h_n}(X_i-x)\frac{\delta_i}{H(T_i)}\Lambda_{b_n}(T_i-y)\!=\!O_p(1)$. Then (5.1) follows from (5.2) and $h_nb_n\rightarrow0$.
\end{proof}

\begin{proof}[Proof of Theorem 3.2]
By (5.4),
\begin{align*}
\Big(f_n(y|X_1),\cdot\cdot\cdot, f_n(y|X_n)\Big)^{\tau}
=&\mathbf{X}\Big(f(y|x),f^{(1,0)}(y|x)\Big)^{\tau}+\frac{f^{(2,0)}(y|x)}{2}\Big((X_1-x)^2,\cdot\cdot\cdot ,(X_n-x)^2\Big)^{\tau}
\nonumber\\&+\Big(o\big((X_1-x)^2\big),\cdot\cdot\cdot ,o\big((X_n-x)^2\big)\Big)^{\tau}+\Big(\frac{b_n^{2}f^{(0,2)}(y|x)}{2}\nabla_{21}+o(b_n^{2})\Big)(1,\cdot\cdot\cdot ,1)^{\tau},
\end{align*}
which combined with (2.2) and (2.3) implies,
\begin{align}
\mathbf{S}^{-1}_n\mathbf{t}^{*}_n=&\mathbf{S}^{-1}_n\mathbf{t}_n-{\rm{diag}}(1,h_n){\Bigg(\begin{array}{c}f(y|x) \\ f^{(1,0)}(y|x)\end{array}\Bigg)}
-\frac{h_n^2f^{(2,0)}(y|x)}{2}\mathbf{S}^{-1}_n{\Bigg(\begin{array}{c}s_{n2}(x)\\s_{n3}(x)\end{array}\Bigg)}+o(h^2_n)\mathbf{S}^{-1}_n{\Bigg(\begin{array}{c}s_{n2}(x)\\s_{n3}(x)\end{array}\Bigg)}\nonumber\\
&-\frac{b_n^2f^{(0,2)}(y|x)}{2}\nabla_{21}\mathbf{S}^{-1}_n{\Bigg(\begin{array}{c}s_{n0}(x)\\s_{n1}(x)\end{array}\Bigg)}+o(b^2_n)\mathbf{S}^{-1}_n{\Bigg(\begin{array}{c}s_{n0}(x)\\s_{n1}(x)\end{array}\Bigg)}
\nonumber\\
=&{\rm{diag}}(1,h_n){\Bigg(\begin{array}{c}\hat{f}_{LL}(y|x)-f(y|x) \\ \hat{f}_{LL}^{(1,0)}(y|x)-f^{(1,0)}(y|x)\end{array}\Bigg)}-\frac{h_n^2f^{(2,0)}(y|x)}{2}\mathbf{S}^{-1}_n{\Bigg(\begin{array}{c}s_{n2}(x)\\s_{n3}(x)\end{array}\Bigg)}+o(h^2_n)\mathbf{S}^{-1}_n{\Bigg(\begin{array}{c}s_{n2}(x)\\s_{n3}(x)\end{array}\Bigg)}\nonumber\\
&-\frac{b_n^2f^{(0,2)}(y|x)}{2}\nabla_{21}\mathbf{S}^{-1}_n{\Bigg(\begin{array}{c}s_{n0}(x)\\s_{n1}(x)\end{array}\Bigg)}+o(b^2_n)\mathbf{S}^{-1}_n{\Bigg(\begin{array}{c}s_{n0}(x)\\s_{n1}(x)\end{array}\Bigg)}.
\end{align}

It follows from the fact that $y<\tau_{_Q}$ and $b_n\rightarrow0$ that there is a $\gamma<\tau_{_Q}$ such that $y+b_n\leq\gamma$ for a large $n$. Then, by assumption (A1), Lemma 5.1 and Lemma 5.2,
\begin{align}
&\sqrt{nh_nb_n}\Big|t^{*}_{nj}(x)-t^{*}_{1nj}(x)\Big|\nonumber\\
\leq&\sqrt{nh_nb_n}\frac{1}{n}\sum\limits_{i=1}^{n}\Big|\frac{X_i-x}{h_n}\Big|^jK_{h_n}(X_i-x)\frac{\delta_i}{H(T_i)}
\Lambda_{b_n}(T_i-y)
\frac{|\hat{H}(T_i)-H(T_i)|}{\hat{H}(T_i)}\nonumber\\
\leq&\sqrt{nh_nb_n}\frac{\sup\limits_{u\in[0,\gamma]} |\hat{H}(u)-H(u)|}{H(\gamma)-\sup\limits_{u\in[0,\gamma]}|\hat{H}(u)-H(u)|}\frac{1}{n}\sum\limits_{i=1}^{n}\Big|\frac{X_i-x}{h_n}\Big|^jK_{h_n}(X_i-x)\frac{\delta_i}{H(T_i)}\Lambda_{b_n}(T_i-y)\nonumber\\
=& \sqrt{nh_nb_n}O_p(n^{-\frac{1}{2}}) \frac{1}{n}\sum\limits_{i=1}^{n}\Big|\frac{X_i-x}{h_n}\Big|^jK_{h_n}(X_i-x)\frac{\delta_i}{H(T_i)}\Lambda_{b_n}(T_i-y)\nonumber\\
=&O_p((h_nb_n)^\frac{1}{2})\frac{1}{n}\sum\limits_{i=1}^{n}\Big|\frac{X_i-x}{h_n}\Big|^jK_{h_n}(X_i-x)\frac{\delta_i}{H(T_i)}\Lambda_{b_n}(T_i-y).
\end{align}
By the same method of (5.5),
$E\Big(\frac{1}{n}\sum\limits_{i=1}^{n}\Big|\frac{X_i-x}{h_n}\Big|^jK_{h_n}\!(X_i-x)\frac{\delta_i}{H(T_i)}\Lambda_{b_n}\!(T_i-y)\Big)\!\!=\!O(1)$, which combined with (5.7) implies that $\frac{1}{n}\sum\limits_{i=1}^{n}\Big|\frac{X_i-x}{h_n}\Big|^jK_{h_n}\!(X_i-x)\frac{\delta_i}{H(T_i)}\Lambda_{b_n}\!(T_i-y)\!=\!O_p(1)$.
Then $\sqrt{nh_nb_n}\Big|t^{*}_{nj}(x)-t^{*}_{1nj}(x)\Big|=O_p((h_nb_n)^\frac{1}{2})$,
which combined with Lemma 5.2 implies
\begin{align}
\sqrt{nh_nb_n}\mathbf{S}^{-1}_n\mathbf{t}^{*}_n=\sqrt{nh_nb_n}\mathbf{S}^{-1}_n\mathbf{t}^{*}_{1n}+O_p((h_nb_n)^\frac{1}{2}).
\end{align}
Hence Theorem 3.2 follows from (5.6), (5.8), Lemma 5.2, Lemma 5.4 and Slutsky's theorem.
\end{proof}

\section{Proofs of lemmas}
\begin{lemma} (Volkonskii 1959)
Let $Z_1,\cdots, Z_m$ be $\alpha$-mixing random variables measurable with respect to $\sigma-algebras$ $\mathcal{F}^{j_1}_{i_1}$,\ldots, $\mathcal{F}^{j_m}_{i_m}$, respectively, with $1\leq i_1<j_1<i_2<j_2<\cdots<j_m\leq n$, $i_{l+1}-j_{l}\geq k\geq 1$ and $|Z_j|\leq 1$ for  $l$,$j=1,2,...,m$. Then
\begin{align*}
\Big|E(\prod\limits_{j=1}^{m}Z_j)-\prod\limits_{j=1}^{m}E(Z_j)\Big|\leq 16(m-1)\alpha(k),
\end{align*}
where $\mathcal{F}_{a}^{b}$=$\sigma\{Z_i{\color{red}|}a \leq i\leq b\}$ and $\alpha(k)$ is the mixing coefficient.
\end{lemma}

\begin{lemma}(Hall and Heyde 1980, Corollary A.2)
Suppose that $Z_1$ and $Z_2$ are random variables such that $E|Z_1|^p<\infty,$ $E|Z_2|^q<\infty,$ where $p,q>1$, $p^{-1}+q^{-1}<1.$ Then
\begin{align*}
|E[Z_1Z_2]-E[Z_1]E[Z_2]| \leq 8(E|Z_1|^p)^{1/p}(E|Z_2|^q)^{1/q} \Big(\sup\limits_{A\in\sigma(Z_1),B\in\sigma(Z_2)}|P(AB)-P(A)P(B)|\Big)^{1-(p^{-1}+q^{-1})}.\end{align*}
\end{lemma}

\begin{proof}[Proof of Lemma 5.3]
Note that
\begin{align}
\tilde{f}_{NW}(y|x)-f(y|x)=&\tilde{f}_{NW}(y|x)-f_n(y|X_i)+f_n(y|X_i)-f(y|x)\nonumber\\
=&s_{n0}^{-1}(x)\times\frac{1}{nh_n}\sum\limits_{i=1}^{n}K\Big(\frac{X_i-x}{h_n}\Big)\left\{\Big[\frac{\delta_i}{H(T_i)}\Lambda_{b_n}(T_i-y)
-f_n(y|X_i)\Big]+\Big[f_n(y|X_i)-f(y|x)\Big]\right\}\nonumber\\
:=&s_{n0}^{-1}(x)[A_n(x)+B_n(x)],
\end{align}
where $A_n(x)\!=\!\frac{1}{nh_n}\sum\limits_{i=1}^{n}K\Big(\frac{X_i-x}{h_n}\Big)\Big[\frac{\delta_i}{H(T_i)}\Lambda_{b_n}(T_i-y)
\!-\!f_n(y|X_i)\Big]$, $B_n(x)\!=\!\frac{1}{nh_n}\sum\limits_{i=1}^{n}K\Big(\frac{X_i-x}{h_n}\Big)\Big[f_n(y|X_i)\!-\!f(y|x)\Big]$.
It follows from Lemma 5.2 that $s_{n0}(x)\stackrel{P}{\longrightarrow}f_X(x)$. Then Lemma 5.3 will follow from (6.1) and Slutsky's theorem if one can prove that
\begin{gather}
B_n(x)\!=\!\frac{h_n^2}{2}\triangle_{21}\Big\{f_X(x)f^{(2,0)}(y|x)\!+\!2f_X^{'}(x)f^{(1,0)}(y|x)\Big\}\!+\!\frac{b_n^2}{2}\nabla_{21}f^{(0,2)}(y|x)f_X(x)
\!+\!o(h_n^2\!+\!b_n^2)\!+\!o_p((nh_n)^{-1/2}),\\
\sqrt{nh_nb_n}A_n(x)\stackrel{D}{\longrightarrow}N\Big(0,\frac{f(x,y)}{H(y)}\nabla_{02}\triangle_{02}\Big).
\end{gather}

First consider (6.2). By the fact that $Z=E(Z)+O_p(\sqrt{Var(Z)})$ for
  any random variable $Z$,  (6.2) will follow if one can prove that  $E[B_n(x)]= \frac{h_n^2}{2}\triangle_{21}\Big\{f_X(x)f^{(2,0)}(y|x)+2f_X^{'}(x)f^{(1,0)}(y|x)\Big\}+\frac{b_n^2}{2}\nabla_{21}f^{(0,2)}(y|x)f_X(x)
+o(h_n^2+b_n^2)$, and $nh_nVar[B_n(x)]=o(1)$.

By (5.4) and assumptions (A1)-(A3),
\begin{align*}
E[B_n(x)]=&\frac{1}{h_n}E\Big[K\Big(\frac{X-x}{h_n}\Big)\Big(f_n(y|X)-f(y|x)\Big)\Big]\\
=&\int_\mathbb{R} \frac{1}{h_n}K\Big(\frac{u-x}{h_n}\Big)\Big[f_n(y|u)-f(y|x)\Big]f_X(u)du\\
=&\int_\mathbb{R} K(v)\Big[h_nf^{(1,0)}(y|x)v+\frac{h_n^2}{2}f^{(2,0)}(y|x)v^2+o(h^2_nv^2)+\frac{b_n^2}{2}\nabla_{21}f^{(0,2)}(y|x)+o(b^2_n)\Big]\nonumber\\&\times
\Big[f_X(x)+h_nf_X^{'}(x)v+o(h_nv)\Big]dv\\
=&\frac{h_n^2}{2}\triangle_{21}\Big[f_X(x)f^{(2,0)}(y|x)+2f_X^{'}(x)f^{(1,0)}(y|x)\Big]+\frac{b_n^2}{2}\nabla_{21}f^{(0,2)}(y|x)f_X(x)
+o(h_n^2+b_n^2).
\end{align*}
It is easy to show that
\begin{align}
nh_nVar[B_n(x)]=&\frac{1}{h_n}Var\Big(K\Big(\frac{X-x}{h_n}\Big)\Big(f_n(y|X)-f(y|x)\Big)\Big)
+\frac{2}{h_n}\sum\limits_{k=1}^{n-1}\Big(1-\frac{k}{n}\Big)\nonumber\\&\times
Cov\Big(K\Big(\frac{X_1-x}{h_n}\Big)\Big(f_n(y|X_1)-f(y|x)\Big),
K\Big(\frac{X_{1+k}-x}{h_n}\Big)\Big(f_n(y|X_{1+k})-f(y|x)\Big)\Big)\nonumber\\
=:& J_{1n}(x)+J_{2n}(x).
\end{align}
From assumptions (A1)-(A3) and (5.4),
\begin{align}
J_{1n}(x)\leq &\frac{1}{h_n}E\Big[K^2\Big(\frac{X-x}{h_n}\Big)\Big(f_n(y|X)-f(y|x)\Big)^2\Big]\nonumber\\
=&\int_\mathbb{R} K^2(v)\Big(f_n(y|x+h_nv)-f(y|x)\Big)^2f_X(x+h_nv)dv=o(1).
\end{align}
From assumptions (A1)-(A4) and (5.4),
\begin{align}
&\Big|Cov\Big(K\Big(\frac{X_1-x}{h_n}\Big)\Big(f_n(y|X_1)-f(y|x)\Big),
K\Big(\frac{X_{1+k}-x}{h_n}\Big)\Big(f_n(y|X_{1+k})-f(y|x)\Big)\Big)\Big|\nonumber\\
\leq&\Big|E\Big(K\Big(\frac{X_1-x}{h_n}\Big)\Big(f_n(y|X_1)-f(y|x)\Big)
K\Big(\frac{X_{1+k}-x}{h_n}\Big)\Big(f_n(y|X_{1+k})-f(y|x)\Big)\Big)\Big|\nonumber\\
&+\Big|E\Big(K\Big(\frac{X_1-x}{h_n}\Big)\Big(f_n(y|X_1)-f(y|x)\Big)\Big)\Big|^2\nonumber\\
\leq& h^2_n\int_\mathbb{R}\int_\mathbb{R} K(u)K(v)\Big|\Big(f_n(y|x+h_nu)-f(y|x)\Big)\Big(f_n(y|x+h_nv)-f(y|x)\Big)\Big| f_k(x+h_nu,x+h_nv)dudv\nonumber\\
&+\Big(h_n\int_\mathbb{R} K(u)\Big(f_n(y|x+h_nu)-f(y|x)\Big)f_X(x+h_nu)du \Big)^2\nonumber\\
=&O(h^2_n).\nonumber
\end{align}
In addition, it follows from Lemma 6.2, (5.4) and assumptions (A1)-(A3) that
\begin{align}
&\Big|Cov\Big(K\Big(\frac{X_1-x}{h_n}\Big)\Big(f_n(y|X_1)-f(y|x)\Big),K\Big(\frac{X_{1+k}-x}{h_n}\Big)\Big(f_n(y|X_{1+k})-f(y|x)\Big)\Big)\Big|\nonumber\\
\leq& Const\, [\alpha(k)]^{1-\frac{1}{\lambda}}\Big(E\Big|K\Big(\frac{X-x}{h_n}\Big)\Big(f_n(y|X)-f(y|x)\Big)\Big|^{2\lambda}\Big)^\frac{1}{\lambda}\nonumber\\
=& Const\, [\alpha(k)]^{1-\frac{1}{\lambda}}\Big(h_n\int_\mathbb{R} |K(u)|^{2\lambda}\Big|f_n(y|x+h_nu)-f(y|x)\Big|^{2\lambda}f_X(x+h_nu)du \Big)^\frac{1}{\lambda}\nonumber\\
=& O\Big([\alpha(k)]^{1-\frac{1}{\lambda}} h^{\frac{1}{\lambda}}_n\Big).\nonumber
\end{align}
Let $d_n=[h^{-(1-\frac{1}{\lambda})/\eta}_n]$, for some $1-\frac{1}{\lambda}<\eta<\lambda-2$, then
\begin{align}
J_{2n}(x)=&O(h^{-1}_n)\Big(\sum\limits_{k=1}^{d_n}+\sum\limits_{k=d_n+1}^{n-1}\Big)\min\left\{h^2_n,
[\alpha(k)]^{1-\frac{1}{\lambda}}h^{\frac{1}{\lambda}}_n\right\}\nonumber\\
=&O(h_nd_n)+O \left(\sum\limits_{k=d_n+1}^{n-1}[\alpha(k)]^{1-\frac{1}{\lambda}}h^{-(1-\frac{1}{\lambda})}_n\right)\nonumber\\
=&O(1)\left\{h_nd_n+d_n^{-(\lambda-2)}h^{-(1-\frac{1}{\lambda})}_n\right\}=O(1)\left\{h_nd_n+h_n^{(1-\frac{1}{\lambda})(\frac{\lambda-2}{\eta}-1)}\right\}
=o(1).
\end{align}
Then, (6.4)-(6.6) imply that $Var[B_n(x)]=o_p((nh_n)^{-1/2})$.

Now consider (6.3). Set $A_{ni}(x)=\Big(\frac{b_n}{h_n}\Big)^{\frac{1}{2}}K\Big(\frac{X_i-x}{h_n}\Big)\Big[\frac{\delta_i}{H(T_i)}\Lambda_{b_n}(T_i-y)
-f_n(y|X_i)\Big],\,1\leq i\leq n$, then $\sqrt{nh_nb_n}A_n(x)=n^{-\frac{1}{2}}\sum\limits_{i=1}^{n}A_{ni}(x)$.
Hence, (6.3) will follow if one can prove
\begin{align}
n^{-\frac{1}{2}}\sum\limits_{i=1}^{n}A_{ni}(x)\stackrel{D}{\longrightarrow}N\Big(0,\frac{f(x,y)}{H(y)}\nabla_{02}\triangle_{02}\Big).
\end{align}

From assumption (A5), there is a sequence of positive integers $\{\delta_n\}$ such that $\delta_n\rightarrow \infty$, \\$\delta_nq_n=o((nh_nb_n)^{\frac{1}{2}})$, and $\delta_n(n(h_nb_n)^{-1})^{\frac{1}{2}}\alpha(q_n)\rightarrow 0$. Let $p_n=[(nh_nb_n)^{\frac{1}{2}}\delta^{-1}_n]$, $r_n=\big[\frac{n}{p_n+q_n}\big]$. Then, it can easily be shown that
\begin{align}
 q_np^{-1}_n\rightarrow 0, \quad r_nq_nn^{-1}\rightarrow 0,\quad r_np_nn^{-1}\rightarrow 1,\quad p_n(nh_nb_n)^{-\frac{1}{2}}\rightarrow 0,\quad r_n\alpha(q_n)\rightarrow 0.
\end{align}
 Next we will make use of Doob's technique which splits the sum $\sum\limits_{i=1}^{n}A_{ni}(x)$ into big and small blocks. Specifically, partition $\{1,2,...,n\}$ into $2 r_n+1$ subsets with large blocks of size $p_n$ and small blocks of size $q_n$.
 Thus, $n^{-\frac{1}{2}}\sum\limits_{i=1}^{n}A_{ni}(x)$ can be written as
 \begin{eqnarray*}
n^{-\frac{1}{2}}\sum\limits_{i=1}^{n}A_{ni}(x)
=n^{-\frac{1}{2}}\{S_{1n}(x)+S_{2n}(x)+S_{3n}(x)\},
\end{eqnarray*}
where
\begin{align*}
S_{1n}(x)=\sum\limits_{j=1}^{r_n}\eta_{1j}(x),\,\,\,S_{2n}(x)=\sum \limits_{j=1}^{r_n}\eta_{2j}(x),\,\,\,S_{3n}(x)=\eta_{3}(x),
\end{align*}
with
\begin{align*}\eta_{1j}(x)=\sum\limits_{i=k_j}^{k_j+p_n-1}A_{ni}(x),\, \,\,\eta_{2j}(x)=\sum\limits_{i=l_j}^{l_j+q_n-1}A_{ni}(x),\,\,\,
\eta_{3}(x)=\sum\limits_{i=r_n(p_n+q_n)+1}^{n}A_{ni}(x),\end{align*}
where $k_j=(j-1)(p_n+q_n)+1$, $l_j=(j-1)(p_n+q_n)+p_n+1$, $j=1,2,...,r_n$.
Then (6.7) will follow if one can prove that
\begin{align}
&n^{-1}E[S^2_{2n}(x)]\rightarrow 0,\quad n^{-1}E[S^2_{3n}(x)]\rightarrow 0,\\
&Var[n^{-\frac{1}{2}}S_{1n}(x)]\rightarrow \frac{f(x,y)}{H(y)}\nabla_{02}\triangle_{02},\\
&\Big|E\exp\Big(it\sum\limits_{j=1}^{r_n}n^{-\frac{1}{2}}\eta_{1j}(x)\Big)-\prod\limits_{j=1}^{r_n}
E\exp\Big(itn^{-\frac{1}{2}}\eta_{1j}(x)\Big)\Big|\rightarrow 0,\\
&e_n(\epsilon)=\frac{1}{n}\sum\limits_{j=1}^{r_n}E[\eta^2_{1j}(x)\mathbf{I}(|\eta_{1j}(x)|>\epsilon\sqrt{n})]\rightarrow 0,\quad \forall\epsilon>0.
\end{align}

First establish (6.9). By (5.3), $E[A_{ni}(x)]=0,\,i=1,\cdots,n$, which imply that $E[S_{in}(x)]=0,\,i=1,2,3$. Then
\begin{align}
n^{-1}E[S^2_{2n}(x)]=n^{-1}Var[S_{2n}(x)]=&\frac{1}{n}\sum\limits_{j=1}^{r_n}\sum\limits_{i=l_j}^{l_j+q_n-1}Var[A_{ni}(x)]
+\frac{2}{n}\sum\limits_{1\leq i<j\leq r_n}^{}Cov(\eta_{2i}(x),\eta_{2j}(x))\nonumber\\
&+\frac{2}{n}\sum\limits_{j=1}^{r_n}\sum\limits_{l_j\leq k<l\leq\l_j+q_n-1}^{}
Cov\Big(A_{nk}(x),A_{nl}(x)\Big)\nonumber\\
=&:J_{3n}(x)+J_{4n}(x)+J_{5n}(x).
\end{align}
For $J_{3n}(x)$, it follows from (5.3) that
\begin{align}
Var[A_{ni}(x)]
=E[A^2_{ni}(x)]=&\frac{b_n}{h_n}E\Big[\Big(\frac{\delta_i}{H(T_i)}\Lambda_{b_n}(T_i-y)-f_n(y|X_i)\Big)^2K^2\Big(\frac{X_i-x}{h_n}\Big)\Big]\nonumber\\
=&\frac{b_n}{h_n}\Big\{E\Big[K^2\Big(\frac{X_i-x}{h_n}\Big)\frac{{\delta_i}^2}{H^2(T_i)}\Lambda_{b_n}^2(T_i-y)\Big]+E\Big[K^2\Big(\frac{X_i-x}{h_n}\Big)f_n^2(y|X_i)\Big]\nonumber\\
&-2E\Big[K^2\Big(\frac{X_i-x}{h_n}\Big)f_n(y|X_i)\frac{\delta_i}{H(T_i)}\Lambda_{b_n}(T_i-y)\Big]\Big\}\nonumber\\
=&\frac{b_n}{h_n}\Big\{E\Big[K^2\Big(\frac{X_i-x}{h_n}\Big)\frac{{\delta_i}^2}{H^2(T_i)}\Lambda_{b_n}^2(T_i-y)\Big]-E\Big[K^2\Big(\frac{X_i-x}{h_n}\Big)f_n^2(y|X_i)\Big]\Big\}\nonumber\\
=&:J_{3n1}(x)+J_{3n2}(x).
\end{align}
From the independence between $(Y_i,X_i)$ and $C_i$, assumptions (A1) and (A2),
\begin{align*}
E\Big[\frac{\delta_i^2}{H^2(T_i)}\frac{1}{b^2_n}\Lambda^2\Big(\frac{T_i-y}{b_n}\Big) \Big|X_i=u\Big]
=&\frac{1}{b^2_n}E\Big[\frac{\mathbf{I}_{\{Y_i\leq C_i\}}}{H^2(Y_i)}\Lambda^2\Big(\frac{Y_i-y}{b_n}\Big) \Big|X_i=u\Big]\nonumber\\
=&\frac{1}{b^2_n}E\Big[\frac{1}{H^2(Y_i)}\Lambda^2\Big(\frac{Y_i-y}{b_n}\Big) E[\mathbf{I}_{\{Y_i\leq C_i\}}|Y_i]\Big|X_i=u\Big]\nonumber\\
=&\frac{1}{b^2_n}E\Big[\frac{1}{H(Y_i)}\Lambda^2\Big(\frac{Y_i-y}{b_n}\Big)\Big|X_i=u\Big]\nonumber\\
=&\frac{1}{b_n}\int_\mathbb{R}\frac{1}{H(y+tb_n)}\Lambda^2(t)f(y+b_nt|u)dt\nonumber\\
=&\frac{1}{b_n}\frac{1}{H(y)}f(y|u)\nabla_{02}+o\Big(\frac{1}{b_n}\Big),
\end{align*}
then, again by assumptions (A1)-(A3),
\begin{align}
J_{3n1}(x)=&\frac{b_n}{h_n}E\Big[K^2\Big(\frac{X_i-x}{h_n}\Big)E\big[\frac{\delta^2_i}{H^2(T_i)}\frac{1}{b^2_n}\Lambda^2
\Big(\frac{T_i-y}{b_n}\Big)\big|X_i\big]\Big]\nonumber\\
=&\frac{b_n}{h_n}\int_\mathbb{R} K^2\Big(\frac{u-x}{h_n}\Big)\Big\{\frac{1}{b_n}\frac{1}{H(y)}f(y|u)\nabla_{02}+o(\frac{1}{b_n})\Big\}f_X(u)du\nonumber\\
=&b_n\int_\mathbb{R} K^2(v)\Big\{\frac{1}{b_n}\frac{1}{H(y)}f(y|x+h_nv)\nabla_{02}+o(\frac{1}{b_n})\Big\}f_X(x+h_nv)dv\nonumber\\
=&\frac{\nabla_{02}}{H(y)}\int_\mathbb{R} K^2(v)f(y|x+h_nv) f_X(x+h_nv)dv+o(1)\nonumber\\
=&\frac{f(x,y)}{H(y)}\nabla_{02}\triangle_{02}+o(1).
\end{align}
In addition, by assumptions (A1)-(A3) and (5.4),
\begin{align}
J_{3n2}(x)&=-\frac{b_n}{h_n}E\Big[K^2\Big(\frac{X_i-x}{h_n}\Big)f_n^2(y|X_i)\Big]\nonumber\\
&=-\frac{b_n}{h_n}\int_{\mathbb{R}} K^2\Big(\frac{u-x}{h_n}\Big)\Big[f(y|x)+O((u-x))+O(b_n^2)\Big]^2f_X(u)du\nonumber\\
&=-b_n\int_{\mathbb{R}} K^2(v)\big[f(y|x)+O(1)h_nv+O(b_n^2)\big]^2f_X(x+h_nv)dv\nonumber\\
&=-b_nO(1)=o(1).
\end{align}
It follows from (6.14)-(6.16) that,
\begin{align}
Var[A_{ni}(x)]=\frac{f(x,y)}{H(y)}\nabla_{02}\triangle_{02}+o(1),
\end{align}
which combined with (6.8) yields that \begin{align}J_{3n}(x)=O(r_nq_nn^{-1})=o(1).\end{align}

Consider $J_{4n}(x)$ and $J_{5n}(x)$. Since both $J_{4n}(x)$ and $J_{5n}(x)$ are bounded by $\frac{2}{n}\sum\limits_{1\leq i<j\leq\ n}|Cov(A_{ni}(x),A_{nj}(x))|$, then,  $J_{4n}(x)=o(1)$ and $J_{5n}(x)=o(1)$ will hold if one can prove that
\begin{align}
\frac{1}{n}\sum\limits_{1\leq i<j\leq\ n}|Cov(A_{ni}(x),A_{nj}(x))|=o(1).
\end{align}
It follows from the fact that $y<\tau_{_Q}$ and $b_n\rightarrow0$ that there is a $\gamma<\tau_{_Q}$ such that $y+b_n\leq\gamma$ for a large $n$. Then, by assumption (A1),
\begin{align}
&\Big|Cov(A_{ni}(x),A_{nj}(x))\Big|\nonumber\\
\leq& \frac{b_n}{h_n}E\Big|K\Big(\frac{X_i-x}{h_n}\Big)\Big(\frac{\delta_i}{H(T_i)}\Lambda_{b_n}(T_i-y)-f_n(y|X_i)\Big)
K\Big(\frac{X_j-x}{h_n}\Big)\Big(\frac{\delta_j}{H(T_j)}\Lambda_{b_n}(T_j-y)-f_n(y|X_j)\Big)\Big|\nonumber\\
\leq&\frac{b_n}{h_n}\frac{1}{H^2(\gamma)}E\Big[\Big|K\Big(\frac{X_i-x}{h_n}\Big)K\Big(\frac{X_j-x}{h_n}\Big)\frac{1}{b_n}
\Lambda\Big(\frac{Y_i-y}{b_n}\Big)\frac{1}{b_n}\Lambda\Big(\frac{Y_j-y}{b_n}\Big)\Big|\Big]\nonumber\\
&+\frac{b_n}{h_n}\frac{1}{H(\gamma)}E\Big[\Big|K\Big(\frac{X_i-x}{h_n}\Big)K\Big(\frac{X_j-x}{h_n}\Big)\frac{1}{b_n}
\Lambda\Big(\frac{Y_i-y}{b_n}\Big)f_n(y|X_j)\Big|\Big]\nonumber\\
&+\frac{b_n}{h_n}\frac{1}{H(\gamma)}E\Big[\Big|K\Big(\frac{X_i-x}{h_n}\Big)K\Big(\frac{X_j-x}{h_n}\Big)\frac{1}{b_n}
\Lambda\Big(\frac{Y_j-y}{b_n}\Big)f_n(y|X_i)\Big|\Big]\nonumber\\
&+\frac{b_n}{h_n}E\Big[\Big|K\Big(\frac{X_i-x}{h_n}\Big)K\Big(\frac{X_j-x}{h_n}\Big)f_n(y|X_i)f_n(y|X_j)\Big|\Big]
\end{align}
From assumptions (A1)-(A4) and (5.4), by simple calculation, \begin{align*}&E\Big[\Big|K\Big(\frac{X_i-x}{h_n}\Big)K\Big(\frac{X_j-x}{h_n}\Big)\frac{1}{b_n}
\Lambda\Big(\frac{Y_i-y}{b_n}\Big)\frac{1}{b_n}\Lambda\Big(\frac{Y_j-y}{b_n}\Big)\Big|\Big]=O(h_n^2),\\
&E\Big[\Big|K\Big(\frac{X_i-x}{h_n}\Big)K\Big(\frac{X_j-x}{h_n}\Big)\frac{1}{b_n}
\Lambda\Big(\frac{Y_i-y}{b_n}\Big)f_n(y|X_j)\Big|\Big]=O(h_n^2),\\
&E\Big[\Big|K\Big(\frac{X_i-x}{h_n}\Big)K\Big(\frac{X_j-x}{h_n}\Big)\frac{1}{b_n}
\Lambda\Big(\frac{Y_j-y}{b_n}\Big)f_n(y|X_i)\Big|\Big]=O(h_n^2),\\
&E\Big[\Big|K\Big(\frac{X_i-x}{h_n}\Big)K\Big(\frac{X_j-x}{h_n}\Big)f_n(y|X_i)f_n(y|X_j)\Big|\Big]=O(h_n^2),
\end{align*}
which combined with (6.20) imply $Cov(A_{ni}(x),A_{nj}(x))=O(h_nb_n)$.
On the other hand, on the base of the fact that there is a $\gamma<\tau_{_Q}$ such that $y+b_n\leq\gamma$ for a large $n$, then, by assumptions (A1)-(A3) and (5.4),
\begin{align}
E|A_{ni}(x)|^{2\lambda}=&\Big(\frac{b_n}{h_n}\Big)^{\lambda}E\Big[\Big|K\Big(\frac{X_i-x}{h_n}\Big)\Big(\frac{\delta_i}{H(T_i)}
\Lambda_{b_n}(T_i-y)-f_n(y|X_i)\Big)\Big|^{2\lambda}\Big]\nonumber\\
\leq&Const\Big(\frac{b_n}{h_n}\Big)^{\lambda}E\Big[\Big|K\Big(\frac{X_i-x}{h_n}\Big)
\frac{1}{H(Y_i)}\Lambda_{b_n}(Y_i-y)\Big|^{2\lambda}\Big]
+Const\Big(\frac{b_n}{h_n}\Big)^{\lambda}E\Big[\Big|K\Big(\frac{X_i-x}{h_n}\Big)f_n(y|X_i)\Big|^{2\lambda}\Big]\nonumber\\
\leq&Const\Big(\frac{b_n}{h_n}\Big)^{\lambda}\frac{h_nb_n}{b^{2\lambda}_nH^{2\lambda}(\gamma)}\int_\mathbb{R}\int_\mathbb{R}\Big|K^{2\lambda}(u)\Lambda^{2\lambda}(v)\Big|
f(x+h_nu,y+b_nv)dudv\nonumber\\
&+Const\Big(\frac{b_n}{h_n}\Big)^{\lambda}h_n\int_\mathbb{R}\Big|K^{2\lambda}(u)\big[f(y|x)+O(1)h_nu+O(b_n^2)\big]^{2\lambda}\Big|f_X(x+h_nu)du\nonumber\\
=&O\Big((h_nb_n)^{1-\lambda}\Big),\nonumber
\end{align}
which combined with Lemma 6.2 implies
\begin{align*}|Cov(A_{n1}(x),A_{n(k+1)}(x))|\leq Const\,[\alpha(k)]^{1-\frac{1}{\lambda}}[E|A_{ni}(x)|^{2\lambda}]^\frac{1}{\lambda}
=O\Big([\alpha(k)]^{1-\frac{1}{\lambda}} (h_nb_n)^{-(1-\frac{1}{\lambda})}\Big).
\end{align*}
Let $c_n=[(h_nb_n)^{-(1-\frac{1}{\lambda})/\eta}]$, for some $1-\frac{1}{\lambda}<\eta<\lambda-2$, then
\begin{align}
\frac{1}{n}\sum\limits_{1\leq i<j\leq\ n}|Cov(A_{ni}(x),A_{nj}(x))|
=&\frac{O(1)}{n}\Big(\sum\limits_{k=1}^{c_n}+\sum\limits_{k=c_n+1}^{n-1}\Big)\min\left\{h_nb_n,
[\alpha(k)]^{1-\frac{1}{\lambda}}(h_nb_n)^{-(1-\frac{1}{\lambda})}\right\}\nonumber\\
=&O(h_nb_nc_n)+O \left(\sum\limits_{k=c_n+1}^{n-1}[\alpha(k)]^{1-\frac{1}{\lambda}}(h_nb_n)^{-(1-\frac{1}{\lambda})}\right)\nonumber\\
=&O(1)\left\{h_nb_nc_n+c_n^{-(\lambda-2)}(h_nb_n)^{-(1-\frac{1}{\lambda})}\right\}\nonumber\\
=&O(1)\left\{(h_nb_n)^{1-(1-\frac{1}{\lambda})/\eta}+(h_nb_n)^{(1-\frac{1}{\lambda})(\frac{\lambda-2}{\eta}-1)}\right\}
=o(1).\nonumber
\end{align}
Then (6.19) follows. And by (6.13), (6.18) and (6.19), $n^{-1}E^2[S^2_{2n}(x)]\rightarrow0$.

Now consider $n^{-1}E[S^2_{3n}(x)]$. It follows from (6.8), (6.17) and (6.19) that
\begin{align}
n^{-1}E[S^2_{3n}(x)]=&n^{-1}Var[S_{3n}(x)]\nonumber\\
=&\frac{1}{n}\sum\limits_{i=r_n(p_n+q_n)+1}^{n}Var[A_{ni}(x)]+\frac{2}{n}\sum\limits_{r_n(p_n+q_n)+1\leq i<j\leq n}^{}Cov(A_{ni}(x),A_{nj}(x))\nonumber\\
\leq& Const\,\frac{n-r_n(p_n+q_n)}{n}+\frac{2}{n}\sum\limits_{1\leq i<j\leq n}^{}\Big|Cov\big(A_{ni}(x),A_{ni}(x)\big)\Big|\nonumber\\
=&o(1).\nonumber
\end{align}
Then (6.9) follows.

Consider (6.10). It follows from (6.8), (6.17) and (6.19) that
\begin{align}
 Var[n^{-\frac{1}{2}}S_{1n}(x)]=&\frac{1}{n}\sum\limits_{j=1}^{r_n}\sum\limits_{i=k_j}^{k_j+p_n-1}Var[A_{ni}(x)]+\frac{2}{n}\sum\limits_{1\leq i<j\leq r_n}^{}Cov(\eta_{1i}(x),\eta_{1j}(x))\nonumber\\
&+\frac{2}{n}\sum\limits_{j=1}^{r_n}\sum\limits_{k_j\leq k<l\leq k_j+p_n-1}^{}
Cov(A_{nk}(x),A_{nl}(x))\nonumber\\
=&\frac{r_np_n}{n}\Big\{\frac{f(x,y)}{H(y)}\nabla_{02}\triangle_{02}+o(1)\Big\}+O\Big(\frac{1}{n}\sum\limits_{1\leq i<j\leq n}^{}\big|Cov\big(A_{ni}(x),A_{ni}(x)\big)\big|\Big)\nonumber\\
=&\frac{f(x,y)}{H(y)}\nabla_{02}\triangle_{02}+o(1),\nonumber
\end{align}
which implies (6.10).

Consider (6.11). By Lemma 6.1 and (6.8),
\begin{align}
\Big|E\exp\Big(it\sum\limits_{j=1}^{r_n}n^{-\frac{1}{2}}\eta_{1j}(x)\Big)-\prod\limits_{j=1}^{r_n}
E\exp\Big(itn^{-\frac{1}{2}}\eta_{1j}(x)\Big)\Big|\leq 16r_n\alpha(q_n+1)\rightarrow 0,\nonumber
\end{align}
which implies (6.11).

Finally, consider (6.12). It follows from that assumption (A1), (5.4) and the fact that there is a $\gamma<\tau_{_Q}$ such that $y+b_n\leq\gamma$ for a large $n$, that $(h_nb_n)^{1/2}A_{ni}(x)=O(1)$. Then $\max\limits_{1\leq j\leq r_n}|\eta_{1j}(x)|=O(p_n(h_nb_n)^{-\frac{1}{2}})$, which combined with the fact that $p_n(nh_nb_n)^{-\frac{1}{2}}\rightarrow 0$ (see (6.8)) implies that for a large enough $j$, \\$\mathbf{I}(|\eta_{1j}(x)|>\epsilon\sqrt{n})=0$. Hence, (6.12) follows.
\end{proof}

\begin{proof}[Proof the Lemma 5.4]
Lemma 5.4 will follow if one can prove that, for any given vector of real numbers $a=(a_0,a_1)^{\tau}\neq 0$,
\begin{align}
\sqrt{nh_nb_n}a^{\tau}t^{*}_{1n}\stackrel{D}{\longrightarrow}N\Big(0,\sigma^2(y|x)f^2_X(x)a^{\tau}\mathbf{V}a\Big).
\end{align}

Let $R_{ni}(x)\!=\!\big(\frac{b_n}{h_n}\big)^{\frac{1}{2}}\tilde{K}\Big(\frac{X_i-x}{h_n}\Big)\!\Big\{\frac{\delta_i}{H(T_i)}\Lambda_{b_n}(T_i-y)
-f_n(y|X_i)\Big\},\,1\!\leq \!i\!\leq\! n$, where $\tilde{K}(u)\!=\!K(u)\{a_0+a_1u\}, u\in\mathbb{R}$. Note that if the function $\tilde{K}(\cdot)$ in $R_{ni}(x)$ is replaced by the function ${K}(\cdot)$ in $A_{ni}(x)$, $R_{ni}(x)$ will turn into $A_{ni}(x)$. Hence, if the function $\tilde{K}(\cdot)$ substitute for the function $K(\cdot)$ in the process of proving (6.7), it can be similarly proved that
\begin{gather}
n^{-\frac{1}{2}}\sum\limits_{i=1}^{n}R_{ni}(x)\stackrel{D}{\longrightarrow}
N\Big(0,\sigma^2(y|x)f^2_X(x)a^{\tau}\mathbf{V}a\Big).
\end{gather}
Then (6.21) follows from (6.22) and the fact that $\sqrt{nh_nb_n}a^{\tau}t^{*}_{1n}=n^{-\frac{1}{2}}\sum\limits_{i=1}^{n}{R_{ni}(x)}$.
\end{proof}

\hspace{-0.6cm}\textbf{References}

\hspace{-0.6cm}A\"{\i}t-Sahalia, Y. 1999.
\newblock Transition densities for interest rate and other non-linear diffusions.
\newblock {\it Journal of Finance} 54(4): 1361-1395.

 \hspace{-0.6cm}Cai, Z.W. 2001.
\newblock Estimating a denstribution function for censored time series data.
\newblock {\it Journal of Multivariate Analysis} 78(2): 299-318.

\hspace{-0.6cm}De Gooijer J.G., and D. Zerom. 2003.
\newblock On conditional density dstimation.
\newblock {\it Statistica Neerlandica} 57(2): 159-176.

\hspace{-0.6cm}Doukhan, P. 1994.
\newblock {\it Mixing: Properties and examples. Lecture notes in statistics,}
\newblock vol. 85. New York: Springer-Verlag.

\hspace{-0.6cm}Eastoe, E.F., C.J. Halsall, J.E. Heffernan, and H. Hung. 2006.
\newblock A statistical comparison
of survival and replacement analyses for the use of censored data in a contaminant air
database: A case study from the Canadian Arctic.
\newblock {\it Atmospheric Environment} 40(34): 6528-6540.

\hspace{-0.6cm}Fan, J., and T.H. Yim. 2004.
\newblock A crossvalidation method for estimating conditional densities.
\newblock {\it Biometrika} 91(4): 819-834.

\hspace{-0.6cm}Ferraty, F., A. Laksaci, and P. Vieu. 2005.
\newblock Functional time series prediction via conditional mode estimation.
\newblock {\it Compte Rendus Acad. Sci. Paris} 340(5): 389-392.

\hspace{-0.6cm}Gussoum, Z., and E. Ould-Sa\"{\i}d. 2008.
\newblock On nonparametric estimation of the regression function under random censorship model.
\newblock {\it Statistics $\&$ Decisions} 26(3): 159-177.

\hspace{-0.6cm}Gussoum, Z., and E. Ould-Sa\"{\i}d. 2010.
\newblock Kernel regression uniform rate estimation for censored data under $\alpha-$mixing condition.
\newblock {\it Electronic Journal of Statistics} 4: 117-132.

\hspace{-0.6cm}Gussoum, Z., and E. Ould-Sa\"{\i}d. 2012.
\newblock Central limit theorem for the kernel estimator of the regression function for censored time series.
\newblock {\it Journal of Nonparametric Statistics} 24(2): 379-397.

\hspace{-0.6cm}Hall, P., and C.C. Heyde. 1980.
\newblock {\it Martingale limit theory and its application}.
\newblock New York: Springer-Verlag.

\hspace{-0.6cm}Hyndman, R.J., D.M. Bashtannyk, and G.K., Grunwald. 1996.
\newblock Estimating and visualizing conditional densities.
\newblock {\it Journal of Computational and Graphical Statistics} 5(4): 315-336.

\hspace{-0.6cm}Izbicki, R., and A.B. Lee. 2017.
\newblock Converting high-dimensional regression to high-dimensional conditional density estimation.
\newblock {\it Electronic Journal of Statistics} 11: 2800-2831

\hspace{-0.6cm}Khardani, S., and S. Semmar. 2014.
\newblock Nonparametric conditional density estimation for censored data based on a recursive kernel.
\newblock {\it Electronic Journal of Statistics} 8(2): 2541-2556. 

\hspace{-0.6cm}Kim, C., M. Oh, S.J. Yang, and H. Choi. 2010.
\newblock A local linear estimation of conditional hazard function in censored data.
\newblock {\it Journal of the Korean Statistical Society} 39(3): 347-355.

\hspace{-0.6cm}Liang, H.Y., and J.I. Baek. 2016.
\newblock Asymptotic normality of conditional density estimation with left-truncated and dependent data.
\newblock {\it Statistical Papers} 57: 1-20.

\hspace{-0.6cm}Liang, H.Y., and J. de U\~{n}a-\'{A}lvarez. 2011
\newblock Asymptotic properties of conditional quantile estimator for censored dependent observations.
\newblock {\it Annals of the Institute of Statistical Mathematics} 63(2): 267-289

\hspace{-0.6cm}Liang, H.Y., and M.C. Iglesias-P¨¦rez. 2018.
\newblock Weighted estimation of conditional mean function with truncated, censored and dependent data.
\newblock {\it Statistics} 52(6):1249-1269.

\hspace{-0.6cm}Liang, H.Y., and L. Peng. 2010.
\newblock Asymptotic normality and Berry-Esseen results for conditional density estimator with censored and dependent data.
\newblock {\it Journal of Multivariate Analysis} 101(5): 1043-1054.

\hspace{-0.6cm}Lipsitz, S.R., and J.G. Ibrahim. 2000.
\newblock Estimation with correlated censored survival data with missing covariates.
\newblock {\it Biostatistics} 1(3): 315-327.

\hspace{-0.6cm}Rachdi, M., A. Laksaci, J. Demongeot, A. Abdali, and F. Madani. 2014.
\newblock Theoretical and practical aspects of the quadratic error in the local linear estimation of the conditional density for functional data.
\newblock {\it Computational Statistics and Data Analysis} 73: 53-68.

\hspace{-0.6cm}Spierdijk, L. 2008.
\newblock Nonparametric conditional hazard rate estimation: A local linear approach.
\newblock {\it Computational Statistics and Data Analysis} 52: 2419-2434.

\hspace{-0.6cm}Stute, W. 1993.
\newblock Almost sure representation of the product-limit estimator for truncated data.
\newblock {\it The Annals of Statistics} 21(1): 146-156.

\hspace{-0.6cm}Volkonskii, V.A., and Y.A. Rozanov. 1959.
\newblock Some limit theorems for random functions.
\newblock {\it Theory of Probability and its Applications} 4(2): 178-197.

\end{spacing}
\end{document}